\documentclass[12pt]{article}
\usepackage{mathtools,amsthm,amsmath,amsfonts,amssymb,tikz-cd,enumitem, graphicx,mathrsfs,bbm}

\tikzcdset{scale cd/.style={every label/.append style={scale=#1},
    cells={nodes={scale=#1}}}}
\usepackage{authblk}
\usepackage{url}
\makeatletter
\newcommand{\address}[1]{\gdef\@address{#1}}
\newcommand{\email}[1]{\gdef\@email{\url{#1}}}
\newcommand{\@endstuff}{\par\vspace{\baselineskip}\noindent\small
\begin{tabular}{@{}l}\@address\\\textit{E-mail address:} \@email\end{tabular}}
\AtEndDocument{\@endstuff}
\makeatother
\usepackage{setspace}
\usepackage[utf8]{inputenc}
\usepackage{CJKutf8}
\counterwithin{equation}{section}
\usepackage[pdfencoding=auto,psdextra,colorlinks=true,linkcolor=blue,citecolor=blue,urlcolor = orange]{hyperref}
\usepackage[margin = 1.05in]{geometry}
\usepackage{kantlipsum}
\allowdisplaybreaks
\newtheorem{theorem}{Theorem}[section]
\newtheorem{definition}[theorem]{Definition}
\newtheorem{example}[theorem]{Example}
\newtheorem{lemma}[theorem]{Lemma}

\newtheorem{remark}[theorem]{Remark}
\newtheorem{proposition}[theorem]{Proposition}

\newcommand{\citep}[1]{\cite{#1}}
\newcommand{\grad}{{\rm grad}}

\newcommand{\dvol}{\text{\normalfont dvol}}
\newcommand{\inv}{\text{\normalfont inv}}

\newcommand{\supp}{\text{\normalfont supp}}

\newcommand{\ind}{\text{\normalfont ind}}
\newcommand{\interiorproduct}{\mathbin{\lrcorner}}
\newcommand{\even}{\text{\normalfont{even}}}
\newcommand{\indtwo}{\text{\normalfont ind}_2^G}
\newcommand{\clifford}{Cl_{0,1}}
\newcommand{\sig}{\text{\normalfont sig}}

\newcommand{\ltwoeven}{L^2(\Lambda^\even T^*M)}
\newcommand{\tsig}{\mathscr{T}_{\sig}}
\newcommand{\dsig}{\mathscr{D}_\sig}
\begin{document}
\title{\textbf{Kervaire semi-characteristics in $\mathbf{KK}$-theory and an Atiyah type vanishing theorem}}
\author{Hao Zhuang}
\address{Department of Mathematics, Washington University in St. Louis}
\email{hzhuang@wustl.edu}
\date{\today}
\maketitle
\begin{abstract}
On $(4n+1)$-dimensional (noncompact) manifolds admitting proper cocompact Lie group actions, we explore the analytic and topological sides of Kervaire semi-characteristics. The analytic side puts together two interpretations, one via assembly maps, and the other via dimensions of kernels. The topological side is ensured by the proper cocompact version of the Hodge theorem. The two sides coincide and admit an Atiyah type vanishing theorem. 
\end{abstract} 
\tableofcontents
\section{Introduction}
Topological invariants on closed manifolds often have insightful index theoretic interpretations. For example, for any $(4n+1)$-dimensional closed oriented manifold $X$, its (topological) Kervaire semi-characteristic $k(X)$ is given by
the dimension of even degree cohomology modulo 2. 
For the analytic side of $k(X)$, in \cite{atiyah2013vector}, Atiyah pointed out that $k(X)$ equals the mod 2 index of $\hat{c}(\dvol)\circ(d+d^*)(1+(d+d^*)^2)^{-1/2}$ on the space of even-degree forms on $X$. Here, the mod 2 index is defined by Atiyah and Singer in \cite{atiyahsingerskewadjoint}: For any skew-adjoint Fredholm operator $F$ on a separable real Hilbert space, 
$\ind_2(F)\coloneqq \dim\ker F\ (\text{mod\ }2)$. It has the name ``index'' because for any compact skew-adjoint $F'$, we have $\ind_2(F+F') = \ind_2(F).$ 

By adding a compact skew-adjoint term to $\hat{c}(\dvol)\circ(d+d^*)(1+(d+d^*)^2)^{-1/2}$, Atiyah proved a vanishing Theorem in \cite[Section 4]{atiyah2013vector}, showing the ``characteristic'' feature of $k(X)$:
\begin{theorem}\label{atiyah vanishing theorem}
    If there are two everywhere independent vector fields on $X$, then $k(X) = 0$. 
\end{theorem}
Is there a generalization of the above theorem to noncompact manifolds? The answer is yes when the manifold is equipped with a proper cocompact Lie group action. There are significant achievements in this direction. In \cite{mathaizhang}, for any manifold admitting a proper cocompact Lie group action, Mathai and Zhang defined the invariant index of $G$-equivariant Dirac type operators. Then, in \cite[Appendix]{mathaizhang}, Bunke provided a $KK$-theoretic interpretation to Mathai and Zhang's index. Following \cite{mathaizhang}, Tang, Yao, and Zhang studied the invariant Hodge theory. They proved in \cite{tangyaozhang} that under a proper cocompact action, the invariant de Rham cohomology is finite dimensional and isomorphic to the kernel of the deformed Laplacian. Also, they showed that when there is an invariant nonvanishing vector field on the manifold, the proper cocompact version of the Euler characteristic equals zero.

Based on the results in \cite{mathaizhang} and \cite{tangyaozhang}, in this paper, we explore the proper cocompact version of the Kervaire semi-characteristic, analytically and topologically, via $KK$-theory and invariant vector fields. 

We first introduce the topological side. Let $M$ be a $(4n+1)$-dimensional oriented smooth manifold without boundary, and $G$ be a Lie group acting properly and cocompactly on $M$ without changing the orientation. We denote by $\chi$ the modular character on $G$ such that $dg^{-1} = \chi(g)dg$ for any left-invariant Haar measure $dg$. Then, we let $H^i_{\chi^{1/2}}(M,d)$ be the $i$-th cohomology group given by the de Rham $d$ and smooth $i$-forms $\omega$ satisfying $g^*\omega = \chi(g)^{1/2}\omega$ for all $g\in G$. 
\begin{definition}\label{true kervaire}
\normalfont
    We call $$k(M,G) \coloneqq \sum_{i\text{\ is\ even}}\dim H^i_{\chi^{1/2}}(M,d)\mod 2$$ the topological Kervaire semi-characteristic of $M$. 
\end{definition}
Our first main result, the Atiyah type vanishing theorem, is stated as follows.
\begin{theorem}\label{topological version of main result and main result 1}
    If there are two $G$-invariant vector fields independent everywhere on $M$, then $k(M,G) = 0$.
\end{theorem}

Although Definition \ref{true kervaire} and Theorem \ref{topological version of main result and main result 1} are topological, we actually take the analytic approach to obtain them. Using Kasparov's $KK$-theory (real version) and the Baum-Connes assembly map, 
we introduce (See Definition \ref{mod 2 index map generalized definition}) a generalized mod 2 index map
$$\ind_2^G: KKO^G(C_0(M),Cl_{0,1})\to KKO(\mathbb{R},Cl_{0,1}),$$
where $C_0(M)$ is the space of continuous functions on $M$ vanishing at infinity, and $Cl_{0,1}$ is the real Clifford algebra generated by $1$ and $v$ satisfying $v^2 = -1$ and $v^* = -v$. 

Back onto the manifold $M$, by assigning a $G$-invariant metric, we get the dual $d^*$ of the de Rham $d$ and let $D = d+d^*$. Given $e_1,\cdots,e_{4n+1}$ an oriented orthonormal local frames, we get an operator $\hat{c}(\dvol) = \hat{c}(e_1)\cdots\hat{c}(e_{4n+1})$ independent of the choice of oriented local orthonormal frame. Then, since $M/G$ is compact, the operator\footnote{The notation $D_\sig$ comes from ``signature''.}  $$D_\sig\coloneqq \hat{c}(\dvol)D(1+D^2)^{-1/2}:  L^2(\Lambda^{\even}T^*M)\to L^2(\Lambda^{\even}T^*M)$$on the space of even-degree $L^2$-forms on $M$ is bounded. It is also skew-adjoint because $\hat{c}(\dvol)$ anticommutes with $D$. Therefore, it defines a class $[\dsig]\in KKO^G(C_0(M),Cl_{0,1})$ (See Section \ref{elliptic section}). 
\begin{definition}\normalfont
    We call $\ind_2^G([\dsig])$ the analytic Kervaire semi-characteristic of $M$. 
\end{definition}


By carrying the procedure in \cite[Appendix E]{mathaizhang} to our situation, we interpret $\ind_2^G([\dsig])$ on a deformed space of forms for the convenience of applying functional calculus. Identifying $KKO(\mathbb{R}, Cl_{0,1})$ with $\mathbb{Z}_2$ (See Proposition \ref{technical proposition}) and using the proper cocompact Hodge theorem \cite[Theorem 1.1]{tangyaozhang}, we have the second main result: 
\begin{theorem}\label{main result 1}
The analytic Kervaire semi-charcteristic $\ind_2^G([\dsig])$ is equal to the topological Kervaire semi-characteristic $k(M,G)$. 
\end{theorem}

Once we prove Theorem \ref{main result 1}, the technique to verify Theorem \ref{topological version of main result and main result 1} is to apply a perturbation similar to the one in \cite[Section 4]{atiyah2013vector} to $D$. This perturbation does not change $\ind_2^G([\dsig])$, but clarifies the dimension (mod 2)
of kernel.




This paper is organized in the following order. In Section \ref{section index map}, we construct the generalized mod 2 index. In Section \ref{hodge theory and garding ineq section}, we give a quick survey of the Garding inequality and the Hodge theorem (proper cocompact version) in \cite{mathaizhang} and \cite{tangyaozhang}. In Sections \ref{elliptic section} and \ref{atiyah's original idea in our proof}, we prove Theorem \ref{main result 1} and then Theorem \ref{topological version of main result and main result 1}. 
In Section \ref{concluding remarks section}, we explain more about Definition \ref{true kervaire} and discuss the influences of the modular character $\chi$ of $G$. 

\ 

\noindent\textbf{Acknowledgments.} I want to sincerely thank my PhD supervisor Prof. Xiang Tang, for encouraging me to explore my very first inspirations more deeply and thoroughly, and for his many helpful suggestions throughout my learning journey in $KK$-theory. I also want to thank Prof. Aliakbar Daemi and Prof. Yanli Song for their practical comments during my first report on this paper. 


\section{Generalized mod 2 index}\label{section index map}
In this section, we define a generalized mod 2 index. We generalize Bunke's steps in \cite[Appendix B]{mathaizhang} on Mathai and Zhang's invariant index to mod 2 index, but with a more refined family of idempotents (and also family of orthogonal projections). 

As in \cite[Section 3]{tangyaozhang}, we let $Y$ be the compact stratified submanifold of $M$ such that $G\cdot Y = M$. After choosing two precompact open subsets $Y\subset U\subset U' \subset M$, we construct a bump function $f: M\to\mathbb{R}$ satisfying $f|_U = 1$ and $\supp(f )\subset U'$. 

Let $dg$ be a left-invariant Haar measure on $G$. With $\chi$ the character map satisfying $dg^{-1} = \chi(g)dg$, a family of the average of $f$ is given by 
\begin{align}\label{twisting function}
    A_t(x)\coloneqq \left(\int_{G} f(gx)^2\chi(g)^tdg\right)^{\frac{1}{2}} \ \ (\forall x\in M)
\end{align}
with a parameter $t\in\mathbb{R}$. We will later see how this $t$ affects Definition \ref{true kervaire}. 

The function $A_t$ has the following pullback property: 
\begin{lemma}\label{A is G-equivariant}
   The function $A_t: M \to\mathbb{R}$ satisfies
    $$A_t(gx) = \chi(g)^{\frac{1-t}{2}} A_t(x)$$
    for all $g\in G$ and $x\in M$. 
\end{lemma}
\begin{proof}
    $A_t(gx)^2 = \displaystyle\int_G f(hgx)^2\chi(h)^tdh = \int_G f(hx)^2\chi(h)^t\chi(g)^{1-t}dh = \chi(g)^{1-t}A_t(x)^2$.  
\end{proof}

With $f$ and $A_t$, we then define the generalized mod 2 index map in $KK$-theory.
\begin{remark}\normalfont
\begin{enumerate}[label = (\arabic*)]
    \item All $C^*$-algebras in this paper are real. Thus, we study real $KK$-groups, denoted by $KKO$. For more properties of real $C^*$-algebras, see \cite{rosenberg2015structureapplicationsrealcalgebras}.
    
    \item As mentioned in \cite[Part 1, Section 12]{baumkknotes}, there are graded $KK^1$ and ungraded $KK^0$, with the interior Kasparov product $KK^i\otimes KK^j\to KK^{i+j}$. However, in this paper, we always take the $\mathbb{Z}_2$-grading into account, and the interior Kasparov product is defined as in \cite{GGKasparov_1981} and \cite{Kasparovequivariant}. 
\end{enumerate}
\end{remark}

We begin with the Baum-Connes assembly map for the skew-adjoint situation. Recall that $Cl_{0,1}$ is the real Clifford algebra generated by $1$ and $v$ under the relations 
$v^2 = -1$ and $v^* = -v.$ The grading on $Cl_{0,1}$ is given by $r+sv\mapsto r-sv$. 

According to \cite[Section 3.11]{Kasparovequivariant}, we have the descent map
$$j^G: KKO^G(C_0(M),\clifford)\to KKO(C^*(G,C_0(M)), C^*(G,\clifford)).$$
Then, with the bump function $f$, we define a family $p_t\in C^*(G, C_0(M))$ by 
$$p_t(g,x) = \dfrac{f(x)}{(A_t(x))^2}\chi(g)^{1-\frac{t}{2}}f(g^{-1}x)$$
for all $g\in G$ and $x\in M$. 
\begin{proposition}
    The family $p_t$ $(t\in\mathbb{R})$ satisfies $p_t^2 = p_t^* = p_t$. 
\end{proposition}
\begin{proof}
    This follows from the definition of crossed product (See \cite[II. Appendix C]{connes1995noncommutative}):  
    \begin{align*}
        p^2_t(g,x) & = \int_G p_t(h,x) p_t(h^{-1}g,h^{-1}x)dh\\
        & = \int_G \dfrac{f(x)}{(A_t(x))^2}\chi(h)^{1-\frac{t}{2}}f(h^{-1}x)\dfrac{f(h^{-1}x)}{(A_t(h^{-1}x))^2}\chi(h^{-1}g)^{1-\frac{t}{2}}f(g^{-1}x)dh\\
        & = \dfrac{f(x)\chi(g)^{1-\frac{t}{2}}f(g^{-1}x)}{A_t(x)^4}\int_Gf(h^{-1}x)^2\chi(h)^{1-t}dh\ \ \ (\text{by\ Lemma \ref{A is G-equivariant}}) \\
        & = p_t(g,x) \ \ \ (\text{by\ the\ definition\ of\ }A_t),
    \end{align*}
    and $p_t^*(g,x) = \chi(g)p(g^{-1},g^{-1}x) = p_t(g,x)$. 
\end{proof}
\begin{remark}\normalfont
    In Bunke's construction, there is a more strict restriction on $f$, so the idempotent family defined in \cite[Appendix B]{mathaizhang} is self-adjoint only when $t = 1$. However, for our $f$, we can make $p_t$ self-adjoint for all $t$.
\end{remark}

Following Lafforgue's interpretation (See \cite{Lafforgue}), we look at the following data: 
\begin{enumerate}[label = (\arabic*)]
    \item A projective $C^*(G,C_0(M))$-module $p_tC^*(G,C_0(M))$ with the trivial grading;
    \item The representation $\rho$ of $\mathbb{R}$ on $p_t C^*(G,C_0(M))$ is the scalar multiplication;
    \item The zero operator $0$ on $p_tC^*(G,C_0(M))$.
\end{enumerate}
Then, we get a class $[p_t]\in KKO(\mathbb{R},C^*(G,C_0(M)))$ which is determined by the triple $(p_tC^*(G, C_0(M)),\rho,0)$.
\begin{remark}\label{independent of choice of bump function}
\normalfont
     The class $[p_t]$ is always independent of $t$. Also, it is independent of the stratified submanifold $Y$, the neighborhoods $U, U'$, and the bump function $f$.
    Similar to \cite[Appendix B]{mathaizhang}, for different $(f, A_t)$ and $(\tilde{f},\tilde{A}_t)$, we have a path $$\gamma(s) = \sqrt{sf^2A_1^{-2}+(1-s)\tilde{f}^2\tilde{A}_1^{-2}}\ \ (0\leq s\leq 1),$$
    and then $p_1(s) \coloneqq \gamma(s) (g^{-1})^*(\gamma(s))\chi(g)^{1/2}$
    gives a homotopy between $[p_1]$ and $[\tilde{p}_1]$. 
\end{remark}
Thus, by Remark \ref{independent of choice of bump function}, the class $[p_t]$ is intrinsically defined. 
Recall the Kasparov descent map $j^G$ (See \cite[Section 3.11]{Kasparovequivariant}) and the interior Kasparov product $\hat{\otimes}_{\text{any\ }C^*\text{-algebra}}$, we state the Baum-Connes assembly map for our $Cl_{0,1}$ situation. 
\begin{definition}\normalfont
    The mod 2 Baum-Connes assembly map $$\mu: KKO^G(C_0(M),\clifford)\to KKO(\mathbb{R},C^*(G,\clifford))$$ is the composition of the Kasparov descent map
    $$j^G: KKO^G(C_0(M),\clifford)\to KKO(C^*(G,C_0(M)), C^*(G,\clifford))$$
    and 
    $$KKO(C^*(G,C_0(M)), C^*(G,\clifford))\xrightarrow{[p_t]\hat{\otimes}_{C^*(G,C_0(M))}\ \cdots\ } KKO(\mathbb{R},C^*(G,\clifford)).$$
\end{definition}

With this mod 2 assembly map $\mu$, we can go one step further to figure out a $\mathbb{Z}_2$-value as the generalized mod 2 index. We let
$[\mathbbm{1}]\in KKO(C^*(G),\mathbb{R})$ be the class
defined by: 
\begin{enumerate}[label = (\arabic*)]
    \item The right $\mathbb{R}$-module $\mathbb{R}$ with the trivial grading;
    \item The representation of $C^*(G)$ on $\mathbb{R}$\ is determined by (See \cite[Section VII.1]{davidson1996c}):
    $$\forall f\in C_c(G) \text{\ and\ } r\in\mathbb{R},\ f\cdot r \coloneqq \left(\int_Gfdg\right)r.$$
    \item The zero operator $0: \mathbb{R}\to \mathbb{R}$.
\end{enumerate}
The definition of the generalized mod 2 index follows immediately.
\begin{definition}\label{mod 2 index map generalized definition}
\normalfont
    The generalized mod 2 index is given by 
    \begin{align*}
        \indtwo: KKO^G(C_0(M),\clifford)&\to KKO(\mathbb{R}, \clifford){\cong}\mathbb{Z}_2\\
        \text{\normalfont any\ class}&\mapsto \mu(\text{\normalfont any\ class})\hat{\otimes}_{C^*(G)}[\mathbbm{1}].
   \end{align*}
\end{definition}

The isomorphism $KKO(\mathbb{R}, \clifford){\cong}\mathbb{Z}_2$ is almost a rephrase of Atiyah and Singer's proof of \cite[Theorem A]{atiyahsingerskewadjoint}, but with grading as an extra requirement: 
    \begin{proposition}\label{technical proposition}
Let $(\mathcal{V}, \rho, \Psi)$ be a triple representing a class in $KKO(\mathbb{R},Cl_{0,1})$, where $\mathcal{V}$ is a graded right Hilbert $Cl_{0,1}$-module, $\rho$ is a graded $C^*$-representation of $\mathbb{R}$ on $\mathcal{V}$, and $\Psi$ is a bounded adjointable operator of degree $1$ on $\mathcal{V}$. Then, the map 
\begin{align*}
        \delta: KKO(\mathbb{R},Cl_{0,1})&\to\mathbb{Z}_2\\
        (\mathcal{V}, \rho, \Psi)&\mapsto\dfrac{1}{2}\dim\ker\left(\rho(1)(\Psi+\Psi^*)\rho(1)\right) \mod 2. 
    \end{align*}
    is an isomorphism. 
\end{proposition}

The image of $\delta$ is exactly Atiyah and Singer's mod 2 index in the case of graded separable Hilbert spaces, and thus $\ind_2^G$ is a reasonable generalization of Atiyah and Singer's $\ind_2$. 


\begin{remark}\normalfont
    According to \cite{roetwoassemblymaps}, for the Fredholm index, there is a coarse algebraic approach to generalize it. Thus, for the mod 2 index map, we may also expect a coarse version. For this direction, see \cite{Bourne_2022}. 
\end{remark}

\section{Garding inequality and Hodge theory}\label{hodge theory and garding ineq section}
In this section, we discuss the operator $T = D+B$, where $D = d+d^*$, and $B$ is an order 0 differential operator exchanges odd and even forms. Moreover, we assume that $B$ is $G$-equivariant. 
We present Mathai and Zhang's Garding inequality of $T$, and Tang, Yao, Zhang's Hodge theorem for $D$. 

According to \cite[Theorem 2.1]{bredon1972introductionliegroupaction}, we assign $M$ a $G$-invariant Riemannian metric $\langle\cdot,\cdot\rangle$. Then, we get the formal adjoint $d^*$ of $d$, the Levi-Civita connection $\nabla$, and the volume form $\dvol$ of $M$.

Next, with the bump function $f$ from Section \ref{section index map}, we define the space. 
$$\Omega_{\chi^{t/2}}^i(M)\coloneqq \{f\cdot\omega: \omega\text{\ is\ a\ smooth\ }i\text{-form\ satisfying\ }g^*\omega = \chi(g)^{t/2}\omega\}.$$ 
Here, our pullback convention is the same as \cite[Chapter 14]{lee2012introduction}, but different from \cite[Appendix]{mathaizhang}. 
Using the Levi-Civita connection $\nabla$, for any $s$, we define the Sobolev $s$-norm  \cite[Definition 4.6]{handbookofglobalanalysis} on the space of differential forms by
$$\|\alpha\|_{s}^2\coloneqq \sum_{j = 0}^s \int_M \langle\nabla^j\alpha, \nabla^j\alpha\rangle \dvol.$$
Then, we let $\mathbf{H }^{s,i}_{\chi^{t/2}}(M,f)$ be the $\|\cdot\|_s$-completion of $f\cdot\Omega^i_\inv(M)$. Respectively, we have $\mathbf{H}^{s,\text{even}}_{\chi^{t/2}}(M,f)$ and $\mathbf{H}^{s}_{\chi^{t/2}}(M,f)$ for even-degree forms and all-degree forms.

Let $L^2(\Lambda^*T^*M)$ be the space of all $L^2$-forms on $M$, and $P_t: L^2(\Lambda^*T^*M)\to L^2(\Lambda^*T^*M)$ be the orthogonal projection onto 
$\mathbf{H }^{s}_{\chi^{t/2}}(M,f)$. We have \cite[Proposition 2.1]{mathaizhang}, the Garding type inequality (or, unbounded Fredholm property): 
\begin{theorem}\label{garding inequalities}
For each $t$, the operator $P_t T: \mathbf{H}^{0}_{\chi^{t/2}}(M,f)\to \mathbf{H}^{0}_{\chi^{t/2}}(M,f)$ is densely defined on $\mathbf{H}^{1}_{\chi^{t/2}}(M,f)$, and there are constants $C_1>0$ and $C_2>0$ satisfying
$$\left\|P_tT(f\eta)\right\|_0\geq C_1\|f\eta\|_1-C_2\|f\eta\|_0 $$
$\text{\ for\ all\ }f\eta\in \mathbf{H}^{1}_{\chi^{t/2}}(M,f)$. Thus, the kernel of $P_tT$ is finite dimensional. 
\end{theorem}

\begin{remark}\normalfont
    According to Bunke's construction in \cite[Appendix B]{mathaizhang}, the space $\mathbf{H}^{s,i}_{\chi^{t/2}}(M,f)$ is actually the image of all locally $\|\cdot\|_s$-integrable $i$-forms $\eta$ satisfying $g^*\eta = \chi^{t/2}(g)\eta$ under the map $\eta\mapsto f\eta$. This is an injective map with a closed image. 
\end{remark}

The expression of $P_t$ has been given in \cite[Appendix D]{mathaizhang} and \cite[Proposition 3.1]{tangyaozhang} for $t = 1$ and $0$ respectively. We refine them into the continuous family $P_t$, which is 
$$P_t(\omega)(x) = \dfrac{f(x)}{A_t(x)^2}\int_G\chi(g)^{\frac{t}{2}}f(gx)\omega(gx)dg,\ \forall \omega\in L^2(\Lambda^*T^*M), x\in M.$$
This $P_t$ is an orthogonal projection from $L^2(\Lambda^*T^*M)$ onto $\mathbf{H}_{\chi^{t/2}}^{0}(M,f)$: 
\begin{lemma}\label{orthogonal projection}
The map 
    $P_t: L^2(\Lambda^*T^*M)\to L^2(\Lambda^*T^*M)$ is an identity map when restricted to $\mathbf{H}_{\chi^{t/2}}^{0}(M,f)$. Also, it satisfies $P_t^* = P_t = P_t^2$, and $g^*(P_t(\omega)/f) = \chi(g)^{t/2}\omega$, $\forall g\in G, \omega\in L^2(\Lambda^*T^*M)$. 
\end{lemma}


Moreover, writing $T$ into the formula $T = \sum_{i = 1}^{4n+1}c(e_i)\nabla_{e_i}+B$, we find: 
\begin{lemma}\label{orthogonal project composition with the dirac}
    For any $\omega\in\Omega^{i}_{\chi^{t/2}}(M)$, 
    $$P_tT(f\omega) = f\cdot (d+d^*+A_t^{-1}dA_t\wedge - A_t^{-1}\grad A_t\interiorproduct)\omega+f\cdot B\omega.$$
\end{lemma}
The verification of Lemma \ref{orthogonal projection} and Lemma \ref{orthogonal project composition with the dirac} follows a similar pattern as \cite[Proposition 3.1, Theorem 4.1]{tangyaozhang}.

Next, we consider the twisted de Rham differentiation
\begin{align}\label{twisted de Rham on chi t/2}
\begin{split}
   d_{A_t}: \Omega_{\chi^{t/2}}^i(M) &\to \Omega_{\chi^{t/2}}^{i+1}(M)\\
   \omega&\mapsto A^{-1}_t(d(A_t\omega)).
   \end{split}
\end{align}
which defines a chain complex. We denote by $H^i_{\chi^{t/2}}(M,d_{A_t})$ the $i$-th cohomology group of chain complex (\ref{twisted de Rham on chi t/2}). Then, following \cite[Section 3]{tangyaozhang}, we look at  
\begin{align*}
    d_{t,f}: f\cdot\Omega^i_{\chi^{t/2}}(M)&\to f\cdot\Omega^{i+1}_{\chi^{t/2}}(M)\\
    f\cdot\omega&\mapsto f\cdot d_{A_t}\omega = f\cdot(d\omega+A^{-1}_tdA_t\wedge\omega).
\end{align*}
Its formal adjoint is given by 
$$d_{t,f}^*(f\omega) = f\cdot (d^*\omega - A^{-1}_t\grad A_t\interiorproduct\omega)$$
for all $\omega\in\Omega^*_{\chi^{t/2}}(M)$. We immediately find that when $B = 0$,
$$P_t D = d_{t,f}+d_{t,f}^*.$$

We have the following Hodge theorem (See \cite[Theorem 1.1]{tangyaozhang}) under the proper cocompact Lie group action: 
\begin{theorem}\label{hodge theorem in noncompact}
    The kernel of $$\left(P_t D\right)^2 = (d_{t,f}+d_{t,f}^*)^2: \mathbf{H}^{0,i}_{\chi^{t/2}}(M,f)\to \mathbf{H}^{0,i}_{\chi^{t/2}}(M,f)$$ is finite dimensional and canonically isomorphic to the cohomology group $H^i_{\chi^{t/2}}(M,d_{A_t})$ through the map $f\eta\mapsto \eta$. 
\end{theorem}

\begin{remark}
   \normalfont In \cite{tangyaozhang}, Theorem \ref{garding inequalities} and Theorem \ref{hodge theorem in noncompact} are proved for $t = 0$. However, the proof for general $t\in\mathbb{R}$ follows the same steps. We leave it to readers as an exercise. 
\end{remark}

\section{Classes by skew-adjoint operators}\label{elliptic section}
In this section, we turn $D$ into a skew-adjoint Fredholm operator and present its image under the mod 2 index map $\ind_2^G$, through a procedure similar to \cite[Appendix E]{mathaizhang}.

Since $D$ is self-adjoint and anti-commutes with $\hat{c}(\dvol)$, we have $$(\hat{c}(\dvol)D)\left(1+(\hat{c}(\dvol)D)^*\hat{c}(\dvol)D\right)^{-1/2} = \hat{c}(\dvol)\circ D(1+D^2)^{-1/2}.$$ Thus, we get a skew-adjoint Fredholm operator
\begin{align}\label{dsig definition}
    D_\sig \coloneqq \hat{c}(\dvol)\circ D(1+D^2)^{-1/2}: L^2(\Lambda^\even T^*M)\to L^2(\Lambda^\even T^*M)
\end{align}
of which the parametrix is $-D_\sig$. 

Now, we consider the following data:
\begin{enumerate}[label = (\arabic*)]
    \item The real Hilbert module $L^2(\Lambda^\even T^*M){\hat{\otimes}}\clifford$ over $Cl_{0,1}$. Its $\mathbb{Z}_2$-grading is given by 
    $\omega+\eta\cdot v\mapsto \omega-\eta\cdot v$.
    \item The representation $\phi_\sig: C_0(M)\to\mathbb{L}(L^2(\Lambda^\even T^*M){\hat{\otimes}}\clifford)$ is given by multiplication of functions.
    \item A self-adjoint bounded Fredholm operator 
    \begin{align*}
        \mathscr{D}_{\sig}: L^2(\Lambda^\even T^*M)\hat{\otimes}\clifford&\to L^2(\Lambda^\even T^*M)\hat{\otimes}\clifford\\
        \omega+\eta v&\mapsto -(D_\sig(\omega)+D_\sig(\eta)v)v.
    \end{align*}
    It is $G$-equivariant as $G$ acts on forms by 
    $g\cdot \omega \coloneqq (g^{-1})^*\omega$. 
\end{enumerate}
Then, $[\dsig]\in KKO^G(C_0(M),\clifford)$ is given by $\left(\ltwoeven{\hat{\otimes}}\clifford, \rho_\sig, \dsig\right)$.

We are ready to compute $\ind^G_2([\dsig])$. First, by \cite[Theorem 3.11]{GGKasparov_1981}, we see that $j^G([\dsig])\in KKO(C^*(G, C_0(M)),C^*(G,Cl_{0,1}))$ is represented by a triple 
$$\left(C^*(G, \ltwoeven\hat{\otimes} Cl_{0,1}), \tilde{\phi}_\sig, \tilde{\mathscr{D}}_\sig\right),$$
with the representation $\tilde{\phi}_\sig$ of 
$C^*(G, C_0(M))$ on $\left(C^*(G, \ltwoeven\hat{\otimes} Cl_{0,1})\right)$
given by 
\begin{align}\label{representation before integral}
    \left(\tilde{\phi}_\sig(\psi)(\omega+\eta v)\right)(h) = \int_G\psi(g) (g^{-1})^*(\omega(g^{-1}h))dg + \int_G\psi(g)(g^{-1})^*(\eta(g^{-1}h))dg\cdot v
\end{align}
for all $\psi\in C_c(G, C_0(M))\text{\ and\ } \omega+\eta v\in C_c(G, \ltwoeven\hat{\otimes} Cl_{0,1})$
(See \cite[Lemma 3.10]{Kasparovequivariant}). 
Moreover, $\tilde{\mathscr{D}}_\sig$ is given by 
\begin{align}\label{before integration}
\left(\tilde{\mathscr{D}}_\sig(\omega+\eta\cdot v)\right)(g) = \dsig\left(\omega(g)+\eta(g)\cdot v\right) = -\left(D_\sig(\omega(g))+D_\sig(\eta(g))\cdot v\right)\cdot v
\end{align}
for all $\omega+\eta\cdot v\in C_c(G, \ltwoeven\hat{\otimes} Cl_{0,1})$. 

Next, in $j^G([\dsig])\hat{\otimes}_{C^*(G)}[\mathbbm{1}]$, the Hilbert module is 
$$C^*(G, \ltwoeven\hat{\otimes} Cl_{0,1})\hat{\otimes}_{C^*(G)}\mathbb{R}.$$
As mentioned in \cite[Section VII.1]{davidson1996c}, the representation of $C^*(G)$ on $\mathbb{R}$ is given by 
$$\forall \psi\in L^1(G)\text{\ and\ } a\in\mathbb{R}, \psi\cdot a = \left(\int_G \psi(g)dg\right)a.$$
Thus, we have an isomorphism
\begin{align*}
    C^*\left(G,\ltwoeven\hat{\otimes} Cl_{0,1}\right)\hat{\otimes}_{Cl_{0,1}\hat{\otimes}C^*(G)}\left(Cl_{0,1}\hat{\otimes}\mathbb{R}\right)&\to \ltwoeven\hat{\otimes} Cl_{0,1}\\
    (\omega + \eta v)\cdot (a+bv) &\mapsto \left(\int_G\omega(g)dg +\int_G\eta(g)dg v\right) (a+bv).
\end{align*}
Since $\dsig$ is $G$-equivariant, the connection (See \cite[Appendix A]{conneslongitudinal}) is given by integrating both sides of (\ref{representation before integral}) and (\ref{before integration}) respectively. Summarizing all computations, we get: 
\begin{lemma}
    The class $j^G([\dsig])\hat{\otimes}[\mathbbm{1}]\in KKO(C^*(G, C_0(M)),Cl_{0,1})$ is represented by the triple
$(\tau_\sig, \ltwoeven\hat{\otimes}Cl_{0,1}, \dsig)$. Here, $\tau_{\sig}$ is determined by  
\begin{align}\label{representation after integral}
    {\tau}_\sig(\psi)(\omega+\eta v) = \int_G\psi(g) (g^{-1})^*\omega dg + \int_G\psi(g)(g^{-1})^*\eta dg\cdot v.
\end{align}
for all $\psi\in C_c(G, C_0(M))\text{\ and\ } \omega+\eta v\in  \ltwoeven\hat{\otimes} Cl_{0,1}$. 
\end{lemma}
By (\ref{representation after integral}), the idempotent $p_t\in C^*(G, C_0(M))$ satisfies $\tau_\sig(p_t)(\omega+\eta v) = P_t\omega + P_t\eta v$ for all $\omega, \eta\in \ltwoeven$. Thus, using the Grassmann connection (See \cite[Proposition A.2]{conneslongitudinal}), we get the triple representing $\ind_2^G([\dsig])$. 
\begin{proposition}
    The class $\ind_2^G([\dsig])\in KKO(\mathbb{R}, Cl_{0,1})$ is represented by the triple 
    $$\left(\mathbf{H}_{\chi^{t/2}}^{0,\text{\normalfont even}}(M,f)\hat{\otimes} Cl_{0,1}, \rho_\sig, \left(P_t\hat{\otimes}\text{\normalfont id}_{Cl_{0,1}}\right)\dsig \left(P_t\hat{\otimes}\text{\normalfont id}_{Cl_{0,1}}\right)\right),$$
    where $\rho_\sig$ is the scalar multiplication of $\mathbb{R}$ on $\mathbf{H}_{\chi^{t/2}}^{0,\text{\normalfont even}}(M,f)\hat{\otimes} Cl_{0,1}$. 
\end{proposition}
\begin{proof}
    Notice that $$p_tC^*(G,C_0(M))\hat{\otimes}_{C^*(G,C_0(M))} \ltwoeven = P_t \ltwoeven. $$
    By Lemma \ref{orthogonal projection}, we get $\mathbf{H}_{\chi^{t/2}}^{0,\text{\normalfont even}}(M,f)$. 
\end{proof}
In addition, by the Garding inequality Theorem \ref{garding inequalities}, we have: 
\begin{lemma}\label{c hat pt d class}
The operator 
    $$\hat{c}(\dvol)\circ P_tD\left(1+(P_tD)^2\right)^{-1/2}: \mathbf{H}_{\chi^{t/2}}^{0,\text{\normalfont even}}(M,f)\to \mathbf{H}_{\chi^{t/2}}^{0,\text{\normalfont even}}(M,f).$$
        is skew-adjoint and Fredholm. 
\end{lemma}
\begin{proof}
    The skew-adjoint property is guaranteed by that $\hat{c}(\dvol)$ commutes with $P_t$, and $(P_t\hat{c}(\dvol)D)^*(P_t\hat{c}(\dvol)D) = (P_tD)^2$. The Fredholm property is by Theorem \ref{garding inequalities}. 
\end{proof}
We let $\varphi = \hat{c}(\dvol)\circ P_tD\left(1+(P_tD)^2\right)^{-1/2}$. With the notations in Proposition {\normalfont\ref{technical proposition}}, we define $\mathcal{V} \coloneqq \mathbf{H}_{\chi^{t/2}}^{0,\text{\normalfont even}}(M,f)\otimes Cl_{0,1}$, $\rho\coloneqq$ scalar multiplication, and $\Psi(x+yv)\coloneqq -(\varphi(x)+\varphi(y)v)v$. They determine a class $[\hat{c}\circ P_tD]\in KKO(\mathbb{R},Cl_{0,1})$. 

At last, to prove Theorem \ref{main result 1}, we should check that 
$$P_tD_\sig P_t-\hat{c}(\dvol)P_tD\left(1+(P_tD)^2\right)^{-1/2}$$ is compact on 
$\mathbf{H}_{\chi^{t/2}}^{0,\text{even}}(M,f)$. Since $\hat{c}(\dvol)$ is bounded and commutes with $P_t$, we only need the following lemma.
\begin{lemma}\label{int representation norm convergent}
The operator
    $$P_tD(1+D^2)^{-1/2}P_t-P_tD(1+(P_tD)^2)^{-1/2}$$ 
is compact on $\mathbf{H}_{\chi^{t/2}}^{0,\text{\normalfont even}}(M,f)$. 
\end{lemma}
\begin{proof}
In fact, we have 
\begin{align*}
& P_tD(1+(P_tD)^2)^{-1/2}-P_tD(1+D^2)^{-1/2}\\
=\ & \int_0^\infty P_tD\left((P_tD)^2+1+s^2)^{-1}-(D^2+1+s^2)^{-1}\right)ds\\
=\ & \int_0^\infty P_tD(D^2+1+s^2)^{-1}\left(D^2-(P_tD)^2\right)((P_tD)^2+1+s^2)^{-1}ds\\
=\ & \int_0^\infty P_tD(D^2+1+s^2)^{-1}\left(D(1-P_t)D+(1-P_t)DP_tD\right)((P_tD)^2+1+s^2)^{-1}ds.
\end{align*}
Notice that $\forall \omega\in\Omega^{\text{even}}_{\chi^{t/2}}(M)$, $$(1-P_t)D(f\omega) = c(df-fA^{-1}dA)\omega.$$
Therefore, $(1-P_t)D|_{\mathbf{H}_{\chi^{t/2}}^{0,\text{even}}(M,f)}$ is bounded. Also, since $P_tD((P_tD)^2+1+s^2)^{-1}$ is compact, the integrand is compact on $\mathbf{H}_{\chi^{t/2}}^{0,\text{even}}(M,f)$. By the following four inequalities
$$\|P_tD((P_tD)^2+1+s^2)^{-1}\|\leq \dfrac{1}{2}(1+s^2)^{-1/2},\ \ \|D(D^2+1+s^2)^{-1}\|\leq \dfrac{1}{2}(1+s^2)^{-1/2},$$
$$\|D(D^2+1+s^2)^{-1}D\|\leq 1,\ \ \|(P_tD)^2((P_tD)^2+1+s^2)^{-1}\|\leq 1,$$
we find that the integral is norm convergent. Therefore, $$P_tD(1+D^2)^{-1/2}P_t-P_tD(1+(P_tD)^2)^{-1/2}$$ 
is compact on $\mathbf{H}_{\chi^{t/2}}^{0,\text{even}}(M,f)$. 
\end{proof}
When $t = 1$, $[\hat{c}\circ P_1D]$ identifies with $\mathbb{Z}_2$ by 
$$\ker\left(\hat{c}(\dvol)\circ P_1D(1+(P_1D)^2)^{-1/2}: \mathbf{H}_{\chi^{t/2}}^{0,\text{even}}(M,f)\to \mathbf{H}_{\chi^{t/2}}^{0,\text{even}}(M,f)\right)\mod 2,$$ which equals $k(M,G)$ according to Theorem \ref{hodge theorem in noncompact}. Thus, Theorem \ref{main result 1} is proved. 

\section{Vanishing theorem}\label{atiyah's original idea in our proof}
In this section, we prove Theorem \ref{topological version of main result and main result 1} using Atiyah's perturbation for Theorem \ref{atiyah vanishing theorem} (See \cite[Section 4]{atiyah2013vector} and \cite[Chapter 7]{wittendeformation} for details). 

Without loss of generality, we assume two $G$-invariant vector fields $V_1$ and $V_2$ on $M$ such that $\langle V_1,V_2\rangle = 0$ and $\langle V_1, V_1\rangle = \langle V_2, V_2\rangle = 1$. 

As in \cite[Section 4]{atiyah2013vector} and \cite[Section 7.2]{wittendeformation}, we perturb $D = d+d^*$ into $$T = \dfrac{1}{2}\left(D-\hat{c}(V_1)\hat{c}(V_2)D\hat{c}(V_2)\hat{c}(V_1)\right).$$ Then, we have the same calculations as \cite[(7.10)]{wittendeformation}: 
\begin{lemma}
    The operator $\hat{c}(\dvol)$ anti-commutes with $T$. Moreover, $T = D+B$ with 
    $$B = \dfrac{1}{2}\sum_{i = 1}^{4n+1}c(e_i)\hat{c}(V_1)\hat{c}(\nabla_{e_i}V_1)+\dfrac{1}{2}\sum_{i=1}^{4n+1}c(e_i)\hat{c}(V_1)\hat{c}(V_2)\hat{c}(\nabla_{e_i}V_2)\hat{c}(V_1).$$
    Thus, $B$ satisfies our assumptions in Section {\normalfont\ref{hodge theory and garding ineq section}}. 
\end{lemma}

Repeating the procedure for $\hat{c}(\dvol)\circ P_tD(1+(P_tD)^2)^{-1/2}$, we see that 
$$\hat{c}(\dvol)\circ P_tT\left(1+(P_tT)^2\right)^{-1/2}: \mathbf{H}_{\chi^{t/2}}^{0,\text{even}}(M,f)\to \mathbf{H}_{\chi^{t/2}}^{0,\text{even}}(M,f)$$
defines a class $[\hat{c}\circ P_tT]\in KKO(\mathbb{R},Cl_{0,1})$ as well. 
\begin{lemma}
The two classes are the same: $[\hat{c}\circ P_tT] = [\hat{c}\circ P_tD]$. 
\end{lemma}
\begin{proof}
    We find the difference 
    \begin{align*}
        & P_tD\left(1+(P_tD)^2\right)^{-1/2} - P_tT\left(1+(P_tT)^2\right)^{-1/2}\\
        =\ & \dfrac{2}{\pi}\int_0^\infty ((P_tD)^2+s^2+1)^{-1}\left(P_tD\cdot P_tB\cdot P_tT-(s^2+1)P_tB\right)(T^2+s^2+1)^{-1}ds. 
    \end{align*}
    Since $P_tB$ is bounded, we just need to apply the same procedure as Lemma \ref{int representation norm convergent} to see this difference is compact. 
\end{proof}
Thus, we get
$$\ind_2^G([\dsig]) = [\hat{c}\circ P_tD] = [\hat{c}\circ P_tT].$$
Using the isomorphism $\delta: KKO(\mathbb{R},Cl_{0,1})\xrightarrow{\cong}\mathbb{Z}_2$, the final step is to calculate 
$$\dim\ker\left(\hat{c}(\dvol)\circ P_tT(1+(P_tT)^2)^{-1/2}: \mathbf{H}_{\chi^{t/2}}^{0,\text{even}}(M,f)\to \mathbf{H}_{\chi^{t/2}}^{0,\text{even}}(M,f)\right) \mod 2.$$
However, the operator $\hat{c}(V_1)\hat{c}(V_2)$ commutes with $\hat{c}(\dvol)\circ P_tT(1+(P_tT)^2)^{-1/2}$ and satisfies $(\hat{c}(V_1)\hat{c}(V_2))^2 = -1$. Thus, it assigns a complex structure to the kernel, making the dimension an even number. Therefore, $\ind_2^G([\dsig]) = 0$, and Theorem \ref{topological version of main result and main result 1} is proved. 
\begin{remark}\normalfont
    In the above proof, we perturb the class $[\hat{c}\circ P_tD]$ instead of $[\dsig]$. It could be very tempting to use $T$ to define a class $[\tsig]\in KKO^G(C_0(M),Cl_{0,1})$ and claim that this class is equal to $[\dsig]$. However, this idea does not work since $\varepsilon\cdot B$ may not be bounded on $\ltwoeven$ for all $\varepsilon\in C_0(M)$. Such an obstacle highlights the effects of $f$ and $P_t$.
\end{remark}

\section{Modular characters and cohomology}\label{concluding remarks section}
In this section, we explain the parameter $t$, how we determine Definition \ref{true kervaire}, and how the modular character $\chi$ is related to the cohomology of $M$. 

Recall the definition of the cohomology group $H^i_{\chi^{t/2}}(M,d_{A_t})$ given in Section \ref{hodge theory and garding ineq section}. The boundary operator $d_{A_t} = A_t^{-1}dA_t$ is twisted by the average function $A_t$. 
The Hodge theory given in Theorem \ref{hodge theorem in noncompact} tells us that the kernel of $P_tD$ on $f\cdot \Omega^i_{\chi^{t/2}}(M)$ is finite dimensional, and is isomorphic to $H^i_{\chi^{t/2}}(M,d_{A_t})$. With $\hat{c}(\dvol)$, the kernel of $\hat{c}(\dvol) P_tD$ on $\mathbf{H}^{0,\even}_{\chi^{t/2}}(M,f)$ is isomorphic to $H^\even_{\chi^{t/2}}(M,d_{A_t})$. Thus, we actually get a semi-characteristic at each $t$: 
\begin{definition}\normalfont
    The $t$-th topological Kervaire semi-characteristic is
    $$k_t(M,G) \coloneqq \sum_{i\text{\ is\ even}} \dim H^i_{\chi^{t/2}}(M,d_{A_t})\ \mod 2,$$
   and all $k_t(M,G)$'s are the same. 
\end{definition}
Here is the question. Is it possible that we use a chain complex not with the twisted
$d_{A_t} = A_t^{-1} d A_t$, but with the original de Rham $d$, to define $k_t(M,G)$? This seems to be an easy question answered by the canonical ``isomorphism''
$$\begin{tikzcd}
	\Omega^i_{\chi^{t/2}}(M) && \Omega^{i+1}_{\chi^{t/2}}(M) \\
 \\
	\Omega^i_{\chi^{t/2}}(M) && \Omega^{i+1}_{\chi^{t/2}}(M)
	\arrow["d", from=1-1, to=1-3]
	\arrow["A_t^{-1}", from=1-1, to=3-1]
	\arrow["A_t^{-1}", from=1-3, to=3-3]
	\arrow["A_t^{-1}dA_t", from=3-1, to=3-3]
\end{tikzcd}\ \ .$$
However, this diagram does not always commute since by Lemma \ref{A is G-equivariant}, for any form $\omega$ satisfying $g^*\omega = \chi(g)^{t/2}\omega$, we have $g^*(A_t^{-1}\omega) = \chi(g)^{t-\frac{1}{2}}\omega$. So, the only case where $A_t$ preserves $\Omega^i_{\chi^{t/2}}(M)$ is $t = 1$. Thus, the topological Kervaire semi-characteristic without twisted boundary map is given in Definition \ref{true kervaire}. 

However, Definition \ref{true kervaire} still engages $\chi$. Shouldn't it be a more intuitive choice to use the invariant cohomology, i.e., the cohomology $H_{\inv}^i(M,d)$ given by forms satisfying $g^*\omega = \omega$ ($\forall g\in G$) with the original de Rham $d$, to define the semi-characteristic? 

As we know, if both $M$ and $G$ are compact, then the invariant cohomology is isomorphic to the original de Rham cohomology. 
Back to our proper cocompact case, we already have $H^i_{\chi^0}(M,d_{A_0})$. Thus, it would be nice if we could replace $d_{A_0}$ by $d$ and let
$$k'(M,G) \coloneqq \sum_{i\text{\ is\ even}}\dim H^i_{\inv}(M,d)\ \mod 2$$
be a more intuitive ``semi-characteristic''. 
However, since $G$ can be non-unimodular, for our $(4n+1)$-dimensional $M$, $k'(M,G)$ does not fit well into the vanishing theorem as $k(M,G)$. 


\begin{example}\label{uni modular non unimodular example comparision}
\normalfont
    Let $\text{Aff}(1)$ be the affine group 
    $$\left\{\begin{bmatrix}
        a&b \\
        0&1
    \end{bmatrix}: a > 0, b \in\mathbb{R}\right\}.$$
    Its Lie algebra is generated by two matrices $X = \begin{bmatrix}
        1&0\\
        0&0
    \end{bmatrix}, Y = \begin{bmatrix}
        0&1\\
        0&0
    \end{bmatrix}.$
    They satisfy $[X, Y] = Y$, so $\text{Aff}(1)$ is non-unimodular (See \cite[Example 7.5.25(c)]{2011structureandgeometryofliegroups}).

    Now, considering the group $$G = \text{Aff}(1)\times\text{Aff}(1)\times\mathbb{R}$$ with the left action on itself. This action is free and proper, and $G$ admits five everywhere independent $G$-invariant vector fields. However, by applying \cite[Proposition 14.29]{lee2012introduction} together with $[X,Y] = Y$, we get
    \begin{align*}
    H^0_{\inv}(G,d) \cong \mathbb{R},
    H^1_{\inv}(G,d) \cong \mathbb{R}^3,
    H^2_{\inv}(G,d) \cong \mathbb{R}^5,\\
    H^3_{\inv}(G,d) \cong \mathbb{R}^4,
        H^4_{\inv}(G,d) \cong \mathbb{R},
    H^5_{\inv}(G,d) = 0.
    \end{align*}
Therefore, we see that $$k'(G,G)=\sum_{i = 0,2,4}\dim H^i_\inv(G,d) = 1 \neq 0\ \mod 2,$$ which does not fit into the vanishing theorem. However, the Euler characteristic is exactly zero in this case. 
\end{example}

The phenomenon in Example \ref{uni modular non unimodular example comparision} is because $H^i_{\inv}(M,d)$ in general does not satisfy the Poincar\'e duality (See the proof of \cite[Theorem 4.1]{tangyaozhang}), which means $k'(M,G)$ cannot represent the even and odd degree parts at the same time, and is not a suitable choice to be the semi-characteristic. However, when $G$ is unimodular, this $k'(M,G)$ is equal to $k(M,G)$, which brings convenience to calculations. 
\begin{example}\label{justifying example for condition}
\normalfont
    We let $G = \mathbb{Z}$ with the discrete topology and the addition operation. Thus, $G$ is a noncompact discrete Lie group. Then, we equip the $5$-dimensional $M = \mathbb{S}^5\times\mathbb{Z}$ with the natural action by integers. By the de Rham cohomology of $\mathbb{S}^5$, we find 
    \begin{align*}
        H^0_{\chi^{1/2}}(M,d) \cong\  & H^5_{\chi^{1/2}}(M,d) \cong \mathbb{R},  \\
    H^1_{\chi^{1/2}}(M,d) \cong\  & H^2_{\chi^{1/2}}(M,d) \cong  H^3_{\chi^{1/2}}(M,d) \cong  H^4_{\chi^{1/2}}(M,d) \cong  0.
    \end{align*}
    Therefore, $k(M,G) = 1$ in this case. 
    
    In fact, the Kervaire semi-characteristic of $\mathbb{S}^5$ (without group actions) is also $1$. By Theorem \ref{atiyah vanishing theorem} in the closed case, $\mathbb{S}^5$ has only one everywhere independent vector field on it. Thus,  $M = \mathbb{S}^5\times\mathbb{Z}$ has only one everywhere independent $\mathbb{Z}$-invariant vector field on it.
\end{example}
Example \ref{justifying example for condition} shows that the assumption of two vector fields in Theorem \ref{topological version of main result and main result 1} is necessary. Meanwhile, we see that the connectedness of $M$ or $G$ is not required. 
\bibliographystyle{abbrv}
\bibliography{mybib.bib}
\end{document}